\documentclass[12pt]{article}

\usepackage{ graphicx, amssymb, pdfpages, 
	amsmath,amsthm,color,mathtools,tikz}
\usetikzlibrary{arrows,calc,intersections}
\usepackage{hyperref,url}
\usepackage{tkz-euclide}

\newtheorem{lemma}{Lemma}

\newtheorem{theorem}{Theorem}
\newtheorem*{theorem*}{Theorem}

\newtheorem{claim}{Claim}

\begin{document}
\newcommand{\eps}{{\varepsilon}}
\newcommand{\proofend}{$\Box$\bigskip}
\newcommand{\C}{{\mathbb C}}
\newcommand{\Q}{{\mathbb Q}}
\newcommand{\R}{{\mathbb R}}
\newcommand{\Z}{{\mathbb Z}}
\newcommand{\RP}{{\mathbb {RP}}}
\newcommand{\CP}{{\mathbb {CP}}}
\newcommand{\Tr}{\rm Tr}
\newcommand{\g}{\gamma}
\newcommand{\G}{\Gamma}
\newcommand{\e}{\varepsilon}
\newcommand{\kk}{\kappa}
\newcommand{\ph}{\hat{p}}
\newcommand{\qh}{\hat{q}}
\newcommand{\rh}{\hat{r}}

\title{Loewner's ``forgotten" theorem}

\author{Peter Albers\footnote{
Mathematisches Institut,
Universit\"at Heidelberg,
69120 Heidelberg,
Germany;
peter.albers@uni-heidelberg.de}
 \and 
 Serge Tabachnikov\footnote{
Department of Mathematics,
Pennsylvania State University,
University Park, PA 16802,
USA;
tabachni@math.psu.edu}
} 

\date{\today}
\maketitle

An oriented smooth closed curve partitions the plane into a number of regions. To every point not on the curve the rotation number is assigned, that is, the number of complete turns that the curve makes about this point. The rotation number is an integer: its sign indicates whether the total rotation is counter-clock or clock-wise.

 These numbers are the same for the points inside one region, but they change when the point crosses the curve, generically by one, see Figure \ref{numbers}.

\begin{figure}[ht]
\centering
\includegraphics[width=.3\textwidth]{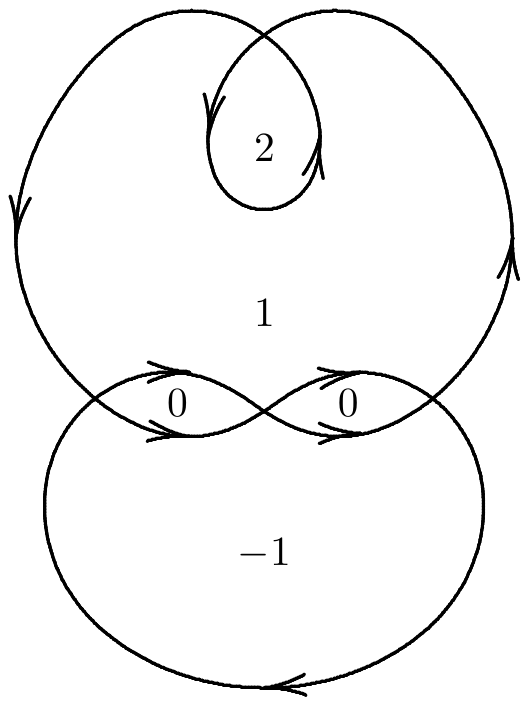}\qquad\qquad\qquad\qquad
\includegraphics[width=.2\textwidth]{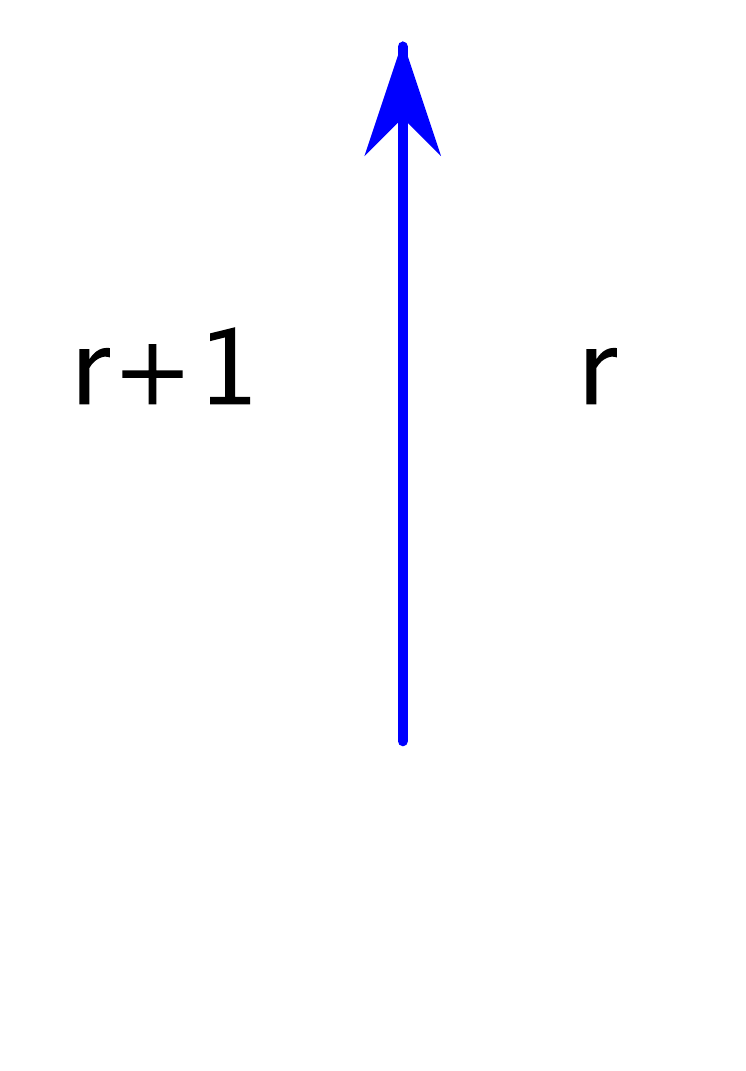}
\caption{On the left: rotation numbers of a specific curve. On the right: ``the wall-crossing" formula, i.e.,~if a point crosses an oriented arc (which is part of the curve) from right to left, the rotation number grows by 1.
}
\label{numbers}
\end{figure}

For a curve  $\g$ and a point $x$ not on $\g$, denote the corresponding rotation number by $r_\g(x)$.
Let $f(t)$, $t\in\R$, be a smooth periodic function, and consider the curve $\g(t)=(f'(t),f(t))$; such curves are called \textit{holonomic}. We assume that $\g$ is immersed, that is, the velocity vector $\g'(t)$ never vanishes.\footnote{As a matter of fact, throughout we will (mostly implicitly) assume further non-degeneracy assumptions, all of which hold after small perturbations of the data involved. We leave it to the careful reader to spell these assumptions out.}

\begin{claim} \label{cl:one}
For every point $x$ not on this curve, one has $r_\g(x)\ge 0$.
\end{claim}

In particular, the curve depicted in Figure \ref{numbers} cannot have a parameterization $(f'(t),f(t))$ for any periodic function $f$.

Claim \ref{cl:one} is the simplest case of a theorem of C. Loewner that, in the full generality, we will formulate later. For now, we prove Claim \ref{cl:one}.

Let us investigate how a holonomic curve may intersect a horizontal line $\{f=c\}$. Assuming that this intersection is transverse, that is, $f'\neq 0$ at the intersection point, there are two possibilities depicted in Figure \ref{intersect}. 

\begin{figure}[ht]
\centering
\includegraphics[width=0.9\textwidth]{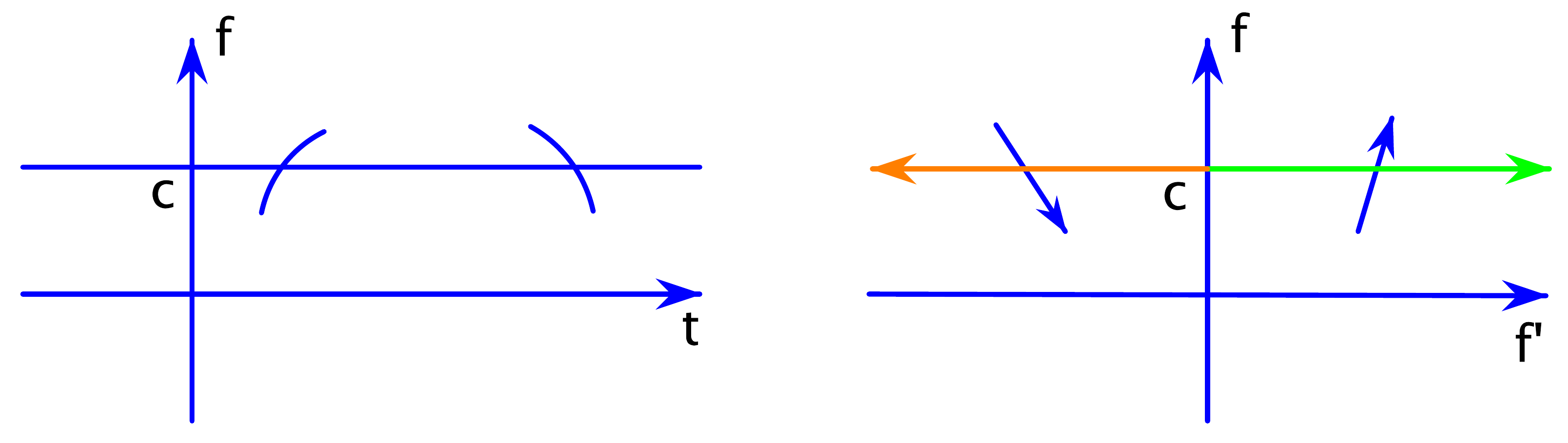}
\caption{Left: intersections of the graph of $f(t)$ with the horizontal line $\{f=c\}$. Right: intersections of the respective holonomic curve with the horizontal line $\{f=c\}$.
}
\label{intersect}
\end{figure}

On the right hand side in Figure \ref{intersect} we consider the horizontal ray given by $\{f'>0, f=c\}$. We notice that the curve $\g=(f',f)$ necessarily intersects this ray ``going upwards'', as indicated. More precisely, if this ray is oriented to the right, the curve $\g=(f',f)$ must intersect this rays from the right side to the left, and not the other way around. The reason simply is that, if $f'(t)>0$, then the value of $f$ is growing near $t$, i.e.,~the second coordinate $f$ of the curve $\g=(f',f)$ is growing. By the same token, if the ray $\{f'<0, f=c\}$ is oriented to the left, the curve $\g=(f',f)$ again must intersect this ray from the right side to the left. 

If we combine this observation with the wall-crossing formula (Figure \ref{numbers}) we conclude the following. If a points moves along the horizontal ray given by $\{f'>0, f=c\}$ out to infinity then, every time it passes the curve, the rotation number around this point drops by $1$. The same happens if a point moves along $\{f'<0, f=c\}$ out to infinity. 

We are ready to prove Claim \ref{cl:one}. Since the curve $\g$ is periodic, it is contained inside a bounded region and therefore the rotation number around points far away from the origin is zero. Now, let $x=(a,b)\notin \g$ and assume for now that $a>0$. Then moving the point out horizontally to the right repeatedly drops the rotation number by $1$ until it eventually becomes zero. Thus, $r_\gamma(x)\geq0$. 

If $a<0$, moving the point horizontally to the left has the same effect of dropping the rotation number eventually down to zero, and again $r_\gamma(x)\geq0$. If $a=0$, move the point a tiny bit left or right and use that the rotation number doesn't change under a small change of $x$. This finishes the proof.

\medskip

It is  worth mentioning another geometric fact: an immersed holonomic curve intersects the vertical axis in the orthogonal direction. Indeed, these points correspond to the critical points of the function $f$, and the orientation of $\g$ at the local minima is to the right, and at the local maxima to the left (since $\g'=(f'',f')$, the second derivative generically does not vanish at the critical points), see Figure \ref{extr}. 

\begin{figure}[ht]
\centering
\includegraphics[width=1\textwidth]{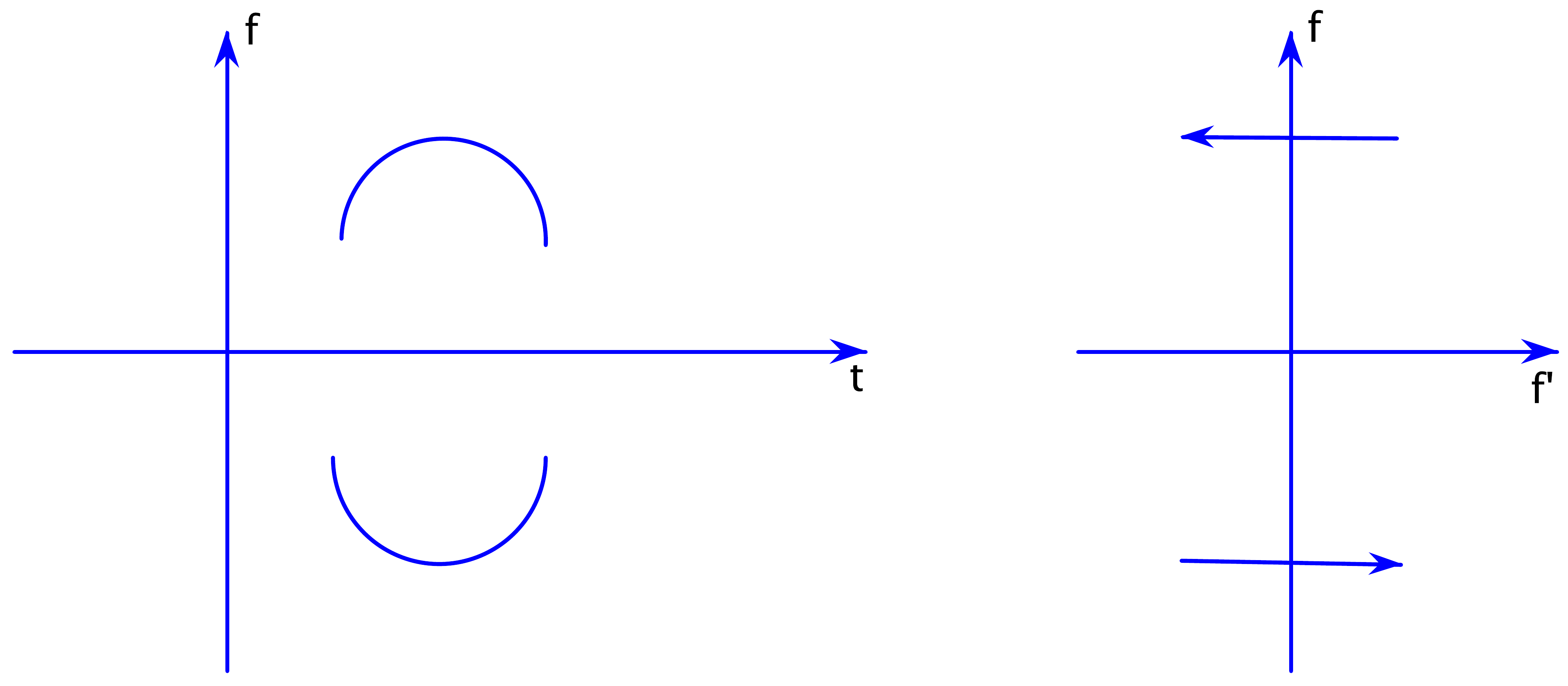}
\caption{Local extrema and the intersections of a holonomic curve with the vertical coordinate axis.
}
\label{extr}
\end{figure}

This implies immediately that the Whitney winding number of $\g$, which by definition is the winding number of $\g'$, equals the number of critical points of $f(t)$. We refer to \cite{Be,BW,EW,Va} for results on holonomic curves and holonomic knots, the knots parameterized as $(f(t),f'(t),f''(t))$. 

Now we present the second simplest particular case of Loewner's theorem; this result and its proof  below are due to G. Bol \cite{Bo}, see also \cite{Kl}.

As before, let $f(t)$ be a periodic function, and now consider the curve $\g(t)=(f''(t)-f(t),f'(t))$ (not to be confused with the curve discussed above). 

\begin{claim} \label{cl:two}
For every point $x$ not on this curve, one has $r_\g(x)\ge 0$.
\end{claim}

\begin{proof}
Let $x=(a,b)\notin \g$ be a point, and assume that $b\le 0$. We compute the rotation number $r_\g(x)$ by counting, with sign, the number of intersections of the curve $\g$ with the horizontal ray $\{(a+s,b)\mid s>0\}$. 

Suppose that the curve intersects this ray in downward direction, that is, moving the point outwards along the ray increases the rotation number by $1$ (as opposed to the situation in the proof of Claim \ref{cl:one}). This poses not a problem, since we will show below that, if we move further along the ray $\{(a+s,b)\mid s>0\}$ to the right, i.e.,~with growing $s$, then there necessarily is another intersection point with $\g$ and this is pointing upwards, see Figure \ref{intersection3}. In particular, the rotation number drops by $1$. 

Thus, when computing the rotation number, every downward intersection of $\g$ with $\{(a+s,b)\mid s>0\}$ is eventually followed by an upward intersection canceling the previous change in rotation number and hence, for $s$ large enough, one has $0=r_\g(a+s,b) \leq r_\g(a,b)$, as we needed to prove. 

\begin{figure}[ht]
\centering
\includegraphics[width=1\textwidth]{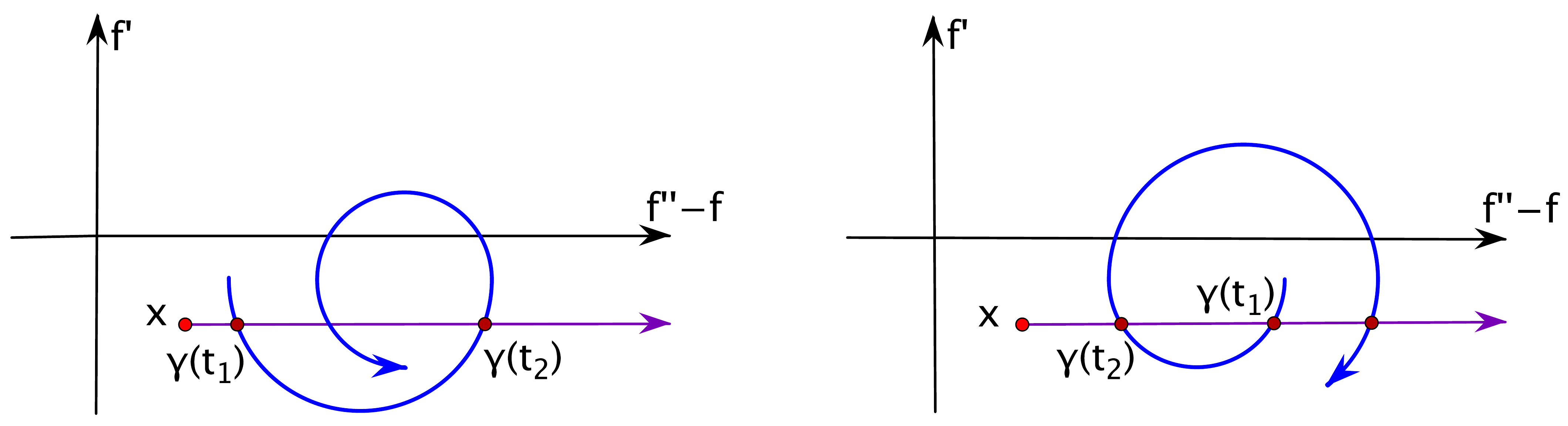}
\caption{Left: a downward intersection of the curve $\g(t)$ with the horizontal ray is followed by an upward intersection farther on the right. Right: an impossible situation.
}
\label{intersection3}
\end{figure}

Thus, we are left with the task of showing that every downward intersection of $\g$ with $\{(a+s,b)\mid s>0\}$ is eventually followed by an upward intersection.  For that let $\g$ intersect the horizontal ray $\{(a+s,b)\mid s>0\}$ at a point $\g(t_1)$ in downward direction. Since $\g$ is a closed curve it needs to intersect the horizontal \textit{line} $\{(a+s,b)\mid s\in\R\}$ in at least one more point. Let $t_2>t_1$ be the smallest parameter with $\g(t_2)\in\{(a+s,b)\mid s\in\R\}$. We will show that $\g(t_2)$ actually lies also on the ray $\{(a+s,b)\mid s>0\}$ and to the right of $\g(t_1)$, and it is an upward intersection, as is shown in Figure \ref{intersection3} on the left. 
 
Consider the auxiliary function $g(t)=f(t)-bt$. We note that $\g(t)=(f''(t)-f(t),f'(t))$ intersects the line $\{(a+s,b)\mid s\in\R\}$ in $t=\tau$ if and only if $g'(\tau)=0$, i.e.,~the intersections are precisely extrema of $g$. Moreover, the curve intersects this ray in the downward direction at $t=\tau$ if and only if $g$ has a strict local maximum in $\tau$ and in downward direction if and only if $g$ has a strict local minimum in $\tau$. 

By assumption $g$ has a local maximum in $t_1$ and, again by assumption, the next extremum of $g$ is in $t_2$. Thus, this needs to be a local minimum. Let us collect this information: $g'(t_1)=g'(t_2)=0$, $g''(t_1)<0$, $g''(t_2)>0$ and $g(t_2) < g(t_1)$. Combining this with $t_2>t_1$ and $b\leq0$, we conclude
$$
f''(t_2)-f(t_2)=g''(t_2)-g(t_2)-bt_2 > g''(t_1)-g(t_1)-bt_1=f''(t_1)-f(t_1).
$$
This is exactly saying that $\g(t_2)$ lies on the ray $\{(a+s,b)\mid s>0\}$ to the right of $\g(t_1)$. So, indeed every downward intersection of $\g$ with $\{(a+s,b)\mid s>0\}$ is eventually followed by an upward intersection.

To finish the proof, note that if $x=(a,b)$ with $b\geq0$, then the same argument works for the horizontal leftward ray $\{(a+s,b)\mid s<0\}$.
\end{proof}

See Figure \ref{prog} for curves illustrating Claims \ref{cl:one} and \ref{cl:two}. 

\begin{figure}[ht]
\centering
\includegraphics[width=.4\textwidth]{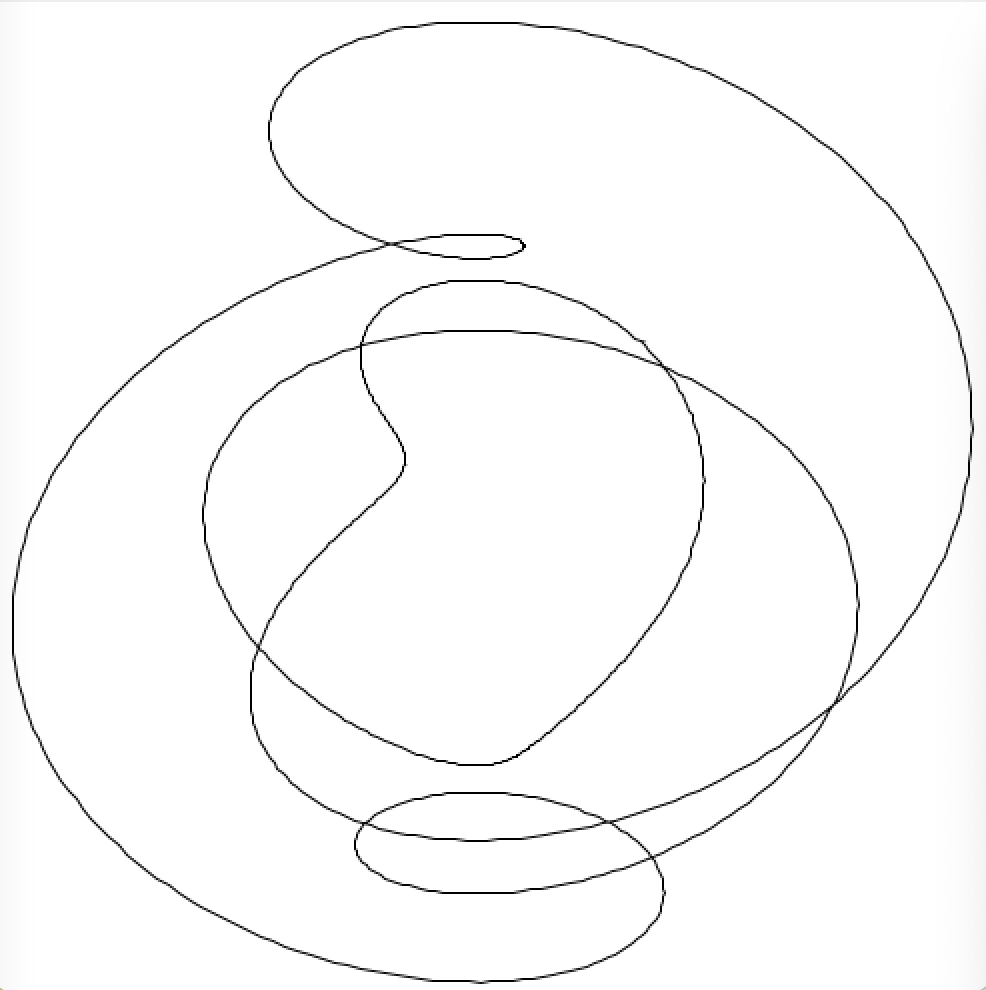}\qquad\qquad
\includegraphics[width=.4\textwidth]{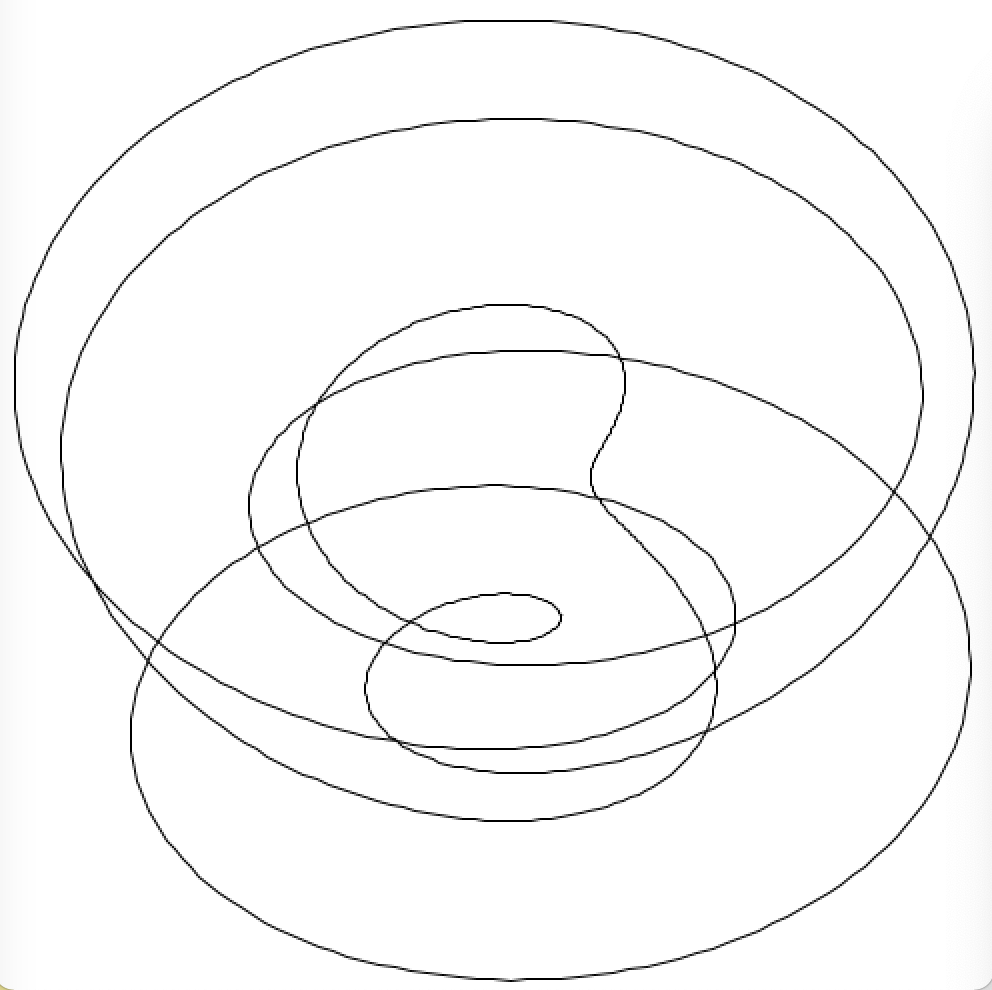}
\caption{The curves $(f'(t),f(t))$ and $(f''(t)-f(t),f'(t))$, where $f(t)$ are random trigonometric polynomials of degree 10, drawn by a computer program created by Richard Schwartz. 
}
\label{prog}
\end{figure}

\subsubsection*{Loewner's theorem}
Let us now present Loewner's full theorem and its (slightly modified) proof taken from \cite{Lo}. Consider a pair of real, monic polynomials
\begin{equation}\label{eqn:polys}
p(x)=\prod_{i=1}^n (x-a_i),\quad q(x)=\prod_{i=1}^{n-1} (x-b_i),
\end{equation}
of degrees $n$ and $n-1$, respectively. 
Assume that for both polynomials all roots are real and  interlaced:
$
a_1<b_1<a_2<b_2<\ldots<b_{n-1}<a_n.
$ 
We point out that, in particular, all roots are simple. Moreover, we also allow the case $n=1$, i.e.,~$p(x)=x-a$ and $q(x)=1$.

Turn the polynomials $p$ and $q$ into differential operators $\ph$ and $\qh$ acting on smooth periodic functions by replacing the variable $x$ by the derivative $\frac{d}{dt}$. That is, $(x-a)$ acts on $f$ as $(\frac{d}{dt}-a)f=f'-af$. Thus, from a  smooth periodic function $f(t)$, we obtain two new functions
$$
F(t)=\ph\big(f(t)\big),\quad G(t)=\qh\big(f(t)\big)
$$
which we combine into a curve $\g(t)=(F(t),G(t))$.

Let us consider two examples. First, if $p(x)=x$, $q(x)=1$, then 
$$
\ph(f(t))=\frac{d}{dt}f(t)=f'(t),\quad\qh(f(t))=1\cdot f(t)=f(t),
$$ 
and $\g=(f',f)$. As the second example, choose $p(x)=x^2-1$, $q(x)=x$, then 
$$
\ph(f(t))=f''(t)-f(t), \quad\qh(f(t))=f'(t),
$$
and $\g=(f''-f,f')$. This way we recover our two examples from above.

Coming back to the situation of general polynomials as in \eqref{eqn:polys}, we consider the curve $\g(t)=(F(t),G(t))$ and let $x$ be a point not on $\g$.

\begin{theorem}[Loewner] \label{thm:L}
For every periodic function $f(t)$, one has $r_\g(x)\ge 0$.
\end{theorem}

\begin{proof}
The proof is by induction on the degree $n$ of the polynomial $p$. The initial step $n=1$ is provided by Claim \ref{cl:one} in the following sense.

If $p(x)=x-a$ and $q(x)=1$, then $\g=(f'-af,f)$. The curve $\g$ is the image of the curve $(f',f)$ under the orientation preserving linear transformation $(x,y)\mapsto (x-ay,y)$ of $\R^2$. Hence both curves have the same rotation number and Claim \ref{cl:one} indeed provides the initial step $n=1$.

The inductive step is based on the next lemma.

\begin{lemma} \label{lm:div}
Divide $p(x)$ by $q(x)$ with remainder:
\begin{equation} \label{eq:root}
p(x)=(x-c)q(x)-r(x).
\end{equation}
Then the pair of polynomials $(q(x),r(x))$  has only real roots, and these roots are interlaced.
\end{lemma}

\begin{proof}
Since $p$ and $q$ are monic polynomials, the degree of $r$ is at most $n-2$, and it suffices to show that $r(x)$ has a root between any two consecutive roots $b_i$ and $b_{i+1}$ of $q(x)$.
 
Since the roots of $p(x)$ are simple, the signs of $p(b_i)$ and $p(b_{i+1})$ are opposite. It follows from (\ref{eq:root}) that the signs of $r(b_i)$ and $r(b_{i+1})$ are opposite as well, therefore $r(x)$ has a root between $b_i$ and $b_{i+1}$.
\end{proof} 

As before, we turn the polynomial $r(x)$ into a differential operator $\rh$ and set $H(t)=\rh(f(t))$.
We consider the curves $\G(t)=(F(t),G(t))$ and $\g(t)=(G(t),H(t))$. 

For a point $x\notin\g\cup\G$, the induction step consists of deforming the curve $\G(t)$ to $\g(t)$ through curves $\G_s(t)$ inside $\R^2$ in such a way that, whenever $\G_s$ moves through the point $x$, the rotation number of $\G_s$ drops by $1$, proving $r_\G(x) \ge r_\g(x)$.

Equation (\ref{eq:root}) allows us to express the function $F$ in terms of $G$ and $H$:
$$
F(t)=G'(t)-cG(t)-H(t).
$$
Deform the curve $\G(t)$ as follows, where $s$ is the parameter of the deformation:
\begin{equation} \label{eq:def} 
\G_s(t):=((1-s)G'(t)-cG(t)-H(t),G(t)),\ 0\le s \le 1.
\end{equation}
We have $\G_0(t)=\G(t)$, and $\G_1(t)=(-cG(t)-H(t),G(t))$, a curve which is the image of the curve $\g(t)=(G(t),H(t))$ under an orientation preserving linear transformation. 

By the induction assumption, $\G_1$ satisfies the assertion of Theorem \ref{thm:L}. It remains to see how the rotation number about the point $x$ changes in the process of deformation.

The rotation number $r_{\Gamma_s}(x)$ changes by $\pm1$ whenever $s$ passes through a value $\sigma$ with $x\in \Gamma_\sigma$, say $x=\Gamma_\sigma(\tau)$. By the wall-crossing formula, the sign of the change of $r_{\Gamma_s}(x)$ equals\footnote{In Figure \ref{numbers} the wall-crossing formula is illustrated for a point moving through a given curve. In the current situation, we move curves through a given point, in this sense, the sign is rule is reversed.} the sign of the determinant made by the tangent vector $\Gamma_\sigma'(\tau)$ of the curve and the velocity vector $\frac{d\G_s(\tau)}{ds}\Big|_{s=\sigma}$ of the ``moving point'', i.e.,~the curve $s\mapsto\Gamma_s(\tau)$. It follows from (\ref{eq:def}) that this determinant is 
$$
\det\left(\Gamma_\sigma'(\tau),\frac{d\G_s(\tau)}{ds}\Big|_{s=\sigma}\right)
=
\begin{pmatrix}
\star &-G'(\tau)\\
G'(\tau) & 0
\end{pmatrix}\\
= (G'(\tau))^2.
$$
Since the determinant is positive, one has $r_\G(x) \ge r_\g(x)$, as needed.
\end{proof}

An instructive exercise is to check that the above argument provides an alternative proof of Claim \ref{cl:two}, deducing it from Claim \ref{cl:one}.

We conclude with a brief description of the context. Loewner's theorem is closely related to Loewner's conjecture that, in turn, is closely related to the Carath\'eodory conjecture that a smooth closed simply connected surface in $\R^3$ has at least two distinct umbilic points, that is, points where the principal curvatures are equal (both conjectures are still open, in spite of numerous attempts on the proofs). See, e.g., \cite{SG} for a survey of the Carath\'eodory conjecture.

The Loewner conjecture concerns the indices of isolated zeros of the planar vector fields whose two components are the real and imaginary parts of the function $(\partial_x+i\partial_y)^n H(x,y)$, where $H(x,y)$ is a smooth function and  $\partial_x+i\partial_y$ is the Cauchy-Riemann operator. The conjecture states that this index does not exceed $n$, see \cite{Tit}. 

Loewner's conjecture implies that of Carath\'edory. Claim \ref{cl:two} was part of Bol's treatment of the Carath\'eodory conjecture in \cite{Bo}. 

\begin{figure}[ht]
\centering
\includegraphics[width=.3\textwidth]{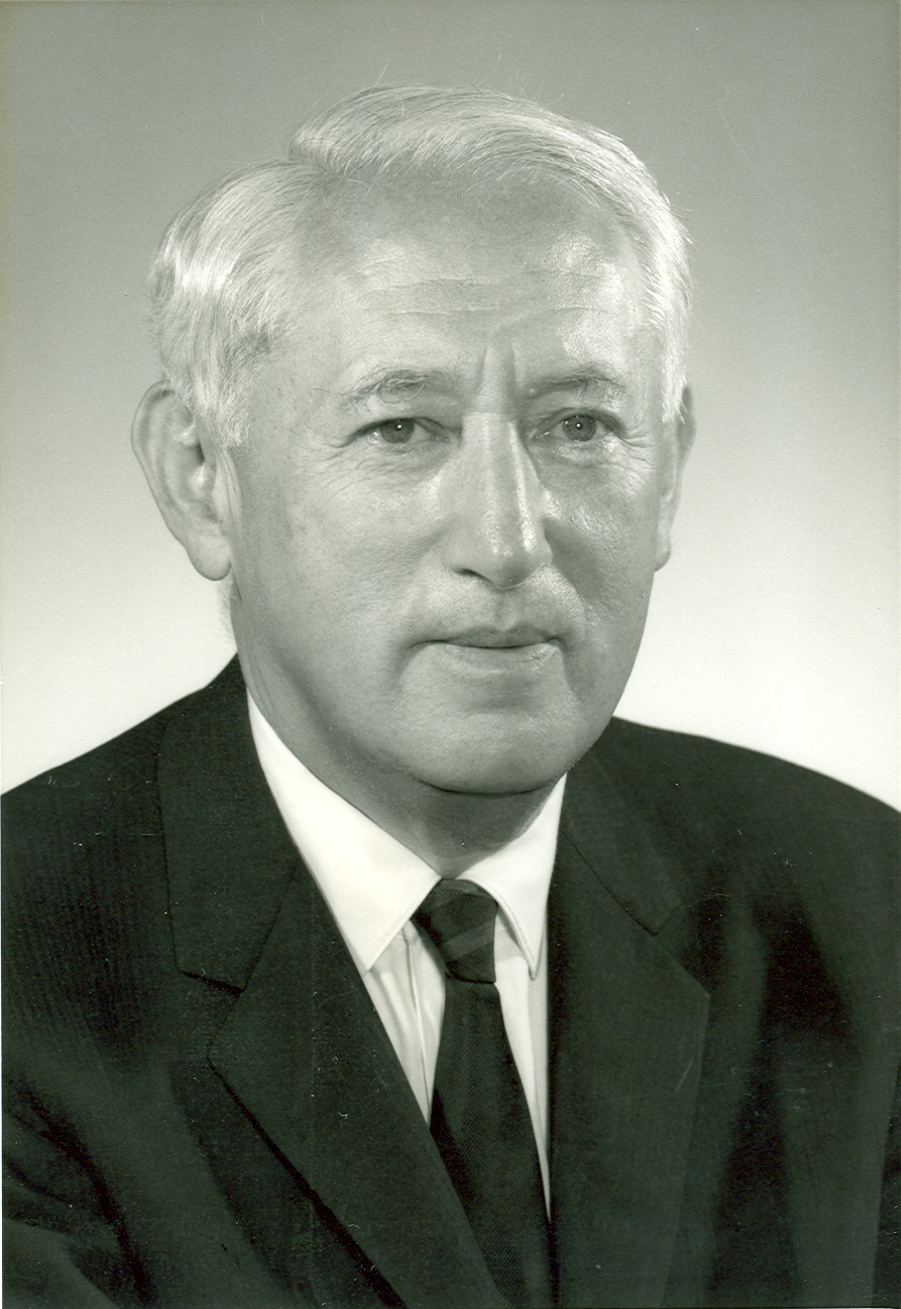}
\caption{Charles Loewner (1893--1968).
}
\label{pic:loewner}
\end{figure}

For biographical information about  Loewner,  we refer to the St. Andrews MacTutor site \cite{St}.

\bigskip

{\bf Acknowledgements}.  We thank Rich Schwartz for creating a program that makes possible to experiment with holonomic curves.
ST is grateful to the Heidelberg University for its invariable hospitality.

PA acknowledge funding by the Deutsche Forschungsgemeinschaft (DFG, German Research Foundation) through Germany’s Excellence Strategy EXC-2181/1 - 390900948 (the Heidelberg STRUCTURES Excellence Cluster), the Transregional Colloborative Research Center CRC/TRR 191 (281071066). ST was supported by NSF grant DMS-2005444 and by a Mercator fellowship within the CRC/TRR 191.


\begin{thebibliography}{99}

\bibitem{Be} M. Bergqvist. {\it Classical invariants for holomic knots.} 
J. Knot Theory Ramifications {\bf 12} (2003),  751--765. 

\bibitem{BW} J. Birman, N. Wrinkle. {\it Holonomic and Legendrian parametrizations of knots.} J. Knot Theory Ramifications {\bf 9} (2000), 293--309.  

\bibitem{Bo} G. Bol. {\it \"Uber Nabelpunkte auf einer Eifl\"ache}. Math. Zeitschrift {\bf 49} (1943-1944), 389--410.

\bibitem{EW} T. Ekholm, M. Wolff. {\it Framed holonomic knots.} Algebr. Geom. Topol. {\bf 2} (2002), 449--463. 

\bibitem{GS} C. Gutierrez, J. Sotomayor. {\it Lines of curvature, umbilic points and Carath\'eodory conjecture.} Resenhas {\bf 3} (1998), 291--322.

\bibitem{Kl} T. Klotz. {\it On G. Bol's proof of Carath\'eodory's conjecture.} Comm. Pure Appl. Math. {\bf 12} (1959), 277--311. 
 
\bibitem{Lo} C. Loewner. {\it A topological characterization of a class of integral operators}.  Ann. of Math. {\bf 49} (1948), 316--332

\bibitem{SG} J. Sotomayor, R. Garcia. {\it  Lines of curvature on surfaces, historical comments and recent developments.} 
S\~ao Paulo J. Math. Sci. {\bf 2} (2008), 99--143. 

\bibitem{Tit} C. Titus. {\it A proof of a conjecture of Loewner and of the conjecture of Caratheodory on umbilic points.} Acta Math. {\bf 131} (1973), 43--77.

\bibitem{Va} V. Vassiliev. {\it Holonomic links and Smale principles for multisingularities.}
J. Knot Theory Ramifications {\bf 6} (1997),  115--123. 

\bibitem{St} \url{https://mathshistory.st-andrews.ac.uk/Biographies/Loewner/}

\end{thebibliography}
\end{document}